\documentclass[times]{oupau}

\begin{document}
\newtheorem{thm}{Theorem}[section]
\newtheorem{lem}[thm]{Lemma}
\newtheorem{defn}[thm]{Definition}
\title{Generalized K\"ahler-Einstein metric along $\mathbb Q$-Fano fibration}


\author{Hassan Jolany\affil{}}

\address{
\affilnum{}University of Lille1, France}

\correspdetails{hassan.jolany@math.univ-lille1.fr}

\received{1 Month 20XX}
\revised{11 Month 20XX}
\accepted{21 Month 20XX}


\begin{abstract}
In this paper, we show that along $\mathbb Q$-Fano fibration, when general fibres, base and central fiber (with at worst Kawamata log terminal singularities)are K-poly stable then there exists a relative K\"ahler-Einstein metric. We introduce the fiberwise K\"ahler-Einstein foliation and we mention that the main difficulty to obtain higher estimates is to solve relative CMA equation along such foliation.

We propose a program such that for finding a pair of canonical metric $(\omega_X,\omega_B)$, which satisfies in 

$Ric(\omega_X)=\pi^*\omega_B+\pi^*(\omega_{WP})+[\mathcal N]$

on K-poly stable degeneration $\pi:X\to B$,  where $Ric(\omega_B)=\omega_B$, we need to have Canonical bundle formula.
\end{abstract}

\maketitle

\section{Introduction}
Let $(X,L)$ be a polarized projective variety. Given an ample line bundle $L\to X$, then a test configuration for the pair $(X,L)$ consists of: 

 - a scheme $\mathfrak X$ with a $\mathbb C^*$-action
 
 - a flat $\mathbb C^*$-equivariant map $\pi:\mathfrak X\to \mathbb C$ with fibres $X_t$; 
 
 - an equivariant line bundle $\mathfrak L\to \mathfrak X$, ample on all fibres;
 
 - for some $r>0$, an isomorphism of the pair $(\mathfrak X_1, \mathfrak L_1)$ with the original pair $(X,L^r)$. 

Let $U_k=H^0(\mathfrak X_0,\mathfrak L_0^k\mid_{\mathfrak X_1})$ be vector spaces with $\mathbb C^*$-action, and let $A_k \colon U_k\to U_k$ be the endomorphisms generating those actions. Then

$$\operatorname{dim} U_k=a_0k^n+a_1k^{n-1}+\dots$$

$$\operatorname{Tr}\left(A_k\right)=b_0k^{n+1}+b_1k^n+\dots$$

Then the Donaldson-Futaki invariant of a test configuration $(\mathfrak X,\mathfrak L)$ is

$$Fut(\mathfrak X,\mathfrak L)=\frac{2(a_1b_0-a_0b_1)}{a_0}.$$

A Fano variety $X$ is K-stable (respectively K-poly stable) if $Fut(\mathfrak X,\mathfrak L)\geq0$ for all normal test configurations $\mathfrak X$ such that $X_0\neq X$ and equality holds if $(\mathcal X,\mathcal L)$ is trivial (respectively, $\mathcal X=X\times \mathbb A^1$).See \cite{5,6,7,8,9} and references therein.

Let $X$ be a projective variety with canonical line bundle $K_X\to X$ of Kodaira dimension $$\kappa(X)=\limsup\frac{\log \dim H^0(X,K_X^{\otimes \ell})}{\log\ell}$$ This can be shown to coincide with the maximal complex dimension of the image of $X$ under pluri-canonical maps to complex projective space, so that $\kappa(X)\in\{-\infty,0,1,...,m\}$.
\;

 A K\"ahler current $\omega$
is called a conical K\"ahler metric (or Hilbert Modular type) with angle $2\pi\beta$, $(0< \beta <1)$ along the divisor $ D$, if
$\omega$ is smooth away from $D$ and
asymptotically equivalent along $D$ to the model conic metric 

$$\omega_{\beta}=\sqrt{-1}\left(\frac{dz_1\wedge d\bar{z_1}}{|z_1|^{2(1-\beta)}}+\sum_{i=2}^n dz_i\wedge d\bar{z_i}\right)$$
 here $(z_1,z_2,...,z_n)$ are local holomorphic coordinates and $D=\{z_1=0\}$ locally. See\cite{18}

For the log-Calabi-Yau fibration $f:(X,D)\to Y$, such that $(X_t,D_t)$ are log Calabi-Yau varieties. If $(X,\omega)$ be a K\"ahler variety with Poincaré singularities then the semi-Ricci flat metric has $\omega_{SRF}|_{X_t}$ is quasi-isometric with the following model which we call it \textbf{fibrewise Poincaré singularities.}

$$\frac{\sqrt[]{-1}}{\pi}\sum_{k=1}^n\frac{dz_k\wedge d\bar {z_k}}{|z_k|^2(\log|z_k|^2)^2}+\frac{\sqrt[]{-1}}{\pi}\frac{1}{\left(\log|t|^2-\sum_{k=1}^n\log|z_k|^2\right)^2}\left(\sum_{k=1}^n\frac{dz_k}{z_k}\wedge \sum_{k=1}^n\frac{d\bar {z_k}}{\bar {z_k}}\right)$$

We can define the same \textbf{fibrewise conical singularities.} and the semi-Ricci flat metric has $\omega_{SRF}|_{X_t}$ is quasi-isometric with the following model

$$\frac{\sqrt[]{-1}}{\pi}\sum_{k=1}^n\frac{dz_k\wedge d\bar {z_k}}{|z_k|^2}+\frac{\sqrt[]{-1}}{\pi}\frac{1}{\left(\log|t|^2-\sum_{k=1}^n\log|z_k|^2\right)^2}\left(\sum_{k=1}^n\frac{dz_k}{z_k}\wedge \sum_{k=1}^n\frac{d\bar {z_k}}{\bar {z_k}}\right)$$

In fact the previous remark will tell us that the semi Ricci flat metric $\omega_{SRF}$ has pole singularities.

A pair $(X,D)$ is log $\mathbb Q$-Fano if the $\mathbb Q$-Cartier divisor $-(K_X+D)$ is ample. For a klt pair $(X,D)$ with
$\kappa(K_X+D)=-\infty$, according to the log minimal model program, there exists a birational map $\phi:X\to Y$
and a morphism
$Y
\to
Z$
such that for
$D'=\phi_*D$ , the pair $(Y_z,D_z')$ is log $\mathbb Q$-Fano with Picard number $\rho(Y_z)=1$ for general $z\in Z$. In particular, log $\mathbb Q$-Fano pairs are the building blocks for pairs with negative Kodaira dimension.

Now, we consider fibrations with K\"ahler Einstein metrics with positive Ricci curvature and CM-line bundle. Let $\pi:(X,D)\to B$ be a holomorphic submersion between compact K\"ahler manifolds of any dimensions, whose log fibers $(X_s,D_s)$ and base have no no-zero holomorphic holomorphic vector fields and whose log fibers admit conical K\"ahler Einstein metrics with positive Ricci curvature on $(X,D)$. We give a sufficient topological condition that involves the CM-line bundle on $B$.

Let $B$ be a normal variety such that $K_B$ is $\mathbb{Q}$-Cartier, and $f:X\to B$ a resolution of singularities. Then, 
$$
K_X = f^\ast(K_B) +\sum_i a_iE_i
$$
where $a_i \in \mathbb{Q}$ and the $E_i$ are the irreducible exceptional divisors.  Then the singularities of $B$ are terminal, canonical, log terminal or log canonical if $a_i > 0, \geq 0, >-1$ or $\geq -1$, respectively.

Firstly, we introduce the log CM-line bundle. Let $\pi:(X,D)\to B$ be a holomorphic submersion between compact K\"ahler manifolds and let $L+D$ be a relatively ample line bundle. Let $\mathcal Y=(X_s,D_s)$ be a log fiber  $n=\dim X_s$ and let $\eta$ denote the constant 

$$\eta=\frac{nc_1(\mathcal Y)c_1(L+D)\mid_{\mathcal Y}^{n-1}}{c_1(L+D)\mid_{\mathcal Y}^{n}}$$

Let $K_{X/B}$ denote the relative canonical bundle. The log CM-line bundle is the virtual bundle and introduced by Tian \cite{39} as follows 

$$\mathcal L_{CM}^D={n+1}\left((K_{X/B}+D)^*-(K_{X/B}+D)\right)\otimes \left((L+D)-(L+D)^*\right)^n-\eta \left((L+D)-(L+D)^*\right)^{n+1} $$

For finding the first Chern classof log CM-line bundle, we need to the following Grothendieck-Riemann-Roch theorem.

Let $X$ be a smooth quasi-projective scheme over a field and $K_{0}(X)$ be the Grothendieck group of bounded complexes of coherent sheaves. Consider the Chern character as a functorial transformation
$${\mbox{ch}}\colon K_{0}(X)\to A(X,{\mathbb {Q}})$$, where $A_{d}(X,{\mathbb {Q}})$
is the Chow group of cycles on $X$ of dimension $d$ modulo rational equivalence, tensored with the rational numbers. 
Now consider a proper morphism

   $$f\colon X\to Y$$

between smooth quasi-projective schemes and a bounded complex of sheaves ${{\mathcal {F}}^{\bullet }}$ on $X$. Let $R$ be the right derived functor, then we have the following theorem of Grothendieck-Riemann-Roch.

\begin{thm}
The Grothendieck-Riemann-Roch theorem relates the pushforward map
    $$f_{\mbox{!}}=\sum (-1)^{i}R^{i}f_{*}\colon K_{0}(X)\to K_{0}(Y)$$

and the pushforward $f_{*}\colon A(X)\to A(Y)$
by the formula
    $${\mbox{ch}}(f_{\mbox{!}}{\mathcal {F}}^{\bullet }){\mbox{td}}(Y)=f_{*}({\mbox{ch}}({\mathcal {F}}^{\bullet }){\mbox{td}}(X))$$
Here $td(X)$ is the Todd genus of (the tangent bundle of) $X$. In fact, since the Todd genus is functorial and multiplicative in exact sequences, we can rewrite the Grothendieck-Riemann-Roch formula as
    $${\mbox{ch}}(f_{\mbox{!}}{\mathcal {F}}^{\bullet })=f_{*}({\mbox{ch}}({\mathcal {F}}^{\bullet }){\mbox{td}}(T_{f}))$$
where $T_f$ is the relative tangent sheaf of $f$, defined as the element $TX - f^*TY$ in $K^0(X)$. For example, when $f$ is a smooth morphism, $T_f$ is simply a vector bundle, known as the tangent bundle along the fibers of $f$.
\end{thm}

So by applying Grothendieck-Riemann-Roch theorem, the first Chern class of log CM-line bundle is $$c_1(\mathcal L_{CM})=2^{n+1}\pi_*\left[\left((n+1)c_1(K_{X/B}+D)+\eta c_1(L+D)\right)c_1(L+D)^n\right]$$

We need to introduce log Weil-Petersson metric on the moduli space of conical K\"ahler-Einstein Fano varieties. We can introduce it by applying Deligne pairing. First we recall Deligne pairing here.\cite{48}

Let
$\pi:X\to S$ be a projective morphism. Let
$\mathcal F$ be a coherent sheaf on
$X$. Let $\det \text{R}\pi_*\mathcal F$ be the line bundle on $S$.
 If $\text{R}^i\pi_*\mathcal F$ is locally free for any $i$, then

$$\det\text{R}\pi_*\mathcal F=\bigotimes_i\left(\det\text{R}^i\pi_*\mathcal F\right)^{(-1)^i} $$

For any flat morphism $\pi:X\to S$ of relative dimension one between normal integral schemes $X$ and $S$, we denote the Deligne pairing by

$$\langle -,-\rangle:\text{Pic}(X)\times \text{Pic}(X)\to\text{Pic}(S)$$

In this case, the Deligne's pairing can be defined as follows. Let
$L_1;L_2\in \text{Pic}(X)$, then

$$\langle L_1,L_2\rangle: \det \text{R}\pi_*(L_1\otimes L_2)\otimes \det \text{R}\pi_*(L_1^{-1})\otimes\text{R}\pi_*(L_2^{-1})\otimes \text{R}\pi_*\mathcal O_X $$

In fact, Deligne pairings provide a method to construct a line bundle over a base $S$ from
line bundles over a fiber space $X$.

Suppose $\pi:X \to S$ is a flat morphism of schemes of relative dimension $n$, i.e., $n=\text{dim} {X_s}$, and suppose $L_i$ are line bundles over the fiber space $X$. Then the Deligne pairing $\langle L_0,...,L_n\rangle_{X/S}$ is a line bundle over $S$. The sections are
given formally by $\langle s_0,...,s_n\rangle$, where each $s_i$ is a rational section of $L_i$ and whose intersection of divisors is empty.

The transition functions between sections are defined 

$$\langle s_0,...,f_js_j,...,s_n\rangle=\mathcal N_f\left[\cap_{i\neq j}\text{div}(s_i)\right]\langle s_0,...,s_j,...,s_n\rangle$$

where $X_p\cap_{i\neq j}\text{div}(s_i)=\sum n_kp_k$, the formal sum of zeros and poles, $\mathcal N_f\left[\cap_{i\neq j}\text{div}(s_i)\right]$ is the product $\prod f(p_k)^{n_k}$ , where the $p_i$ are the zeros and poles in the common intersection,
and $n_k$ the multiplicities.

Let $\pi: (X ,D)\to S$ be a projective family of canonically polarized varieties. Equip
the relative canonical bundle $K_{X/S}+D$ with the hermitian metric that is induced by the fiberwise K\"ahler-Einstein metrics. The log
Weil-Petersson form is equal, up to a numerical factor, to the fiber integral

$$\omega_{WP}^D=\int_{X_s\setminus D_s}c_1\left(K_{X'/S}\right)^{n+1}=\left(\int_{X_s\setminus D_s}|A|_{\omega_s}^2\right)ds\wedge d\bar s$$

$A$ represents the Kodaira-Spencer class of the deformation. See \cite{20}

In fact if we take $\pi: (X,D)\to S$ be a projective family of canonically polarized varieties. Since, the Chern form of the metric on $\langle L_0, . . . , L_n\rangle$ equals the fiber integral \cite{20},
$$\int_{X/S}c_1(L_0,h_0)\wedge...\wedge c_1(L_n,h_n) $$

The curvature of the metric on the Deligne pairing $\langle K_{X/S}+D, . . . , K_{X/S}+D\rangle$
given by the fiberwise K\"ahler-Einstein metric coincides with the generalized log Weil-Petersson form $\omega_{WP}^D$ on $S$.

Take $$\alpha=\frac{c_1(\mathcal L_{CM}^D)}{2^{n+1}(n+1)\pi_*\left(c_1(L+D)^n\right)}\in H^{1,1}(S)$$

In Song-Tian program when the log-Kodaira dimension $\kappa(X,D)$ is $-\infty$ then we don't have canonical model and by using Mori fibre space we can find canonical metric(with additional K-poly stability assumption on fibres ), the following theorem give an answer for finding canonical metric by using Song-Tian program in the case when the log-Kodaira dimension $\kappa(X,D)$ is $-\infty$. We can introduce the logarithmic Weil-Petersson metric on moduli space of log Fano varieties.

Let $\mathcal M^n$ be
the space of all $n$-dimensional K\"ahler-Einstein Fano manifolds, normalized so that the K\"ahler
form
$\omega$ is in the class $2\pi c_1(X)$, modulo biholomorphic isometries.It 
is pre-compact in the Gromov-Hausdorff topology, see \cite{40}. Donaldson and Sun \cite{23}, introduced the refined Gromov-Hausdorff compactification
$\overline{\mathcal M}^n$ of $\mathcal M^n$ such that every point in the boundary
$\overline{\mathcal M}^n\setminus \mathcal M^n$ is naturally a $\mathbb Q$-Gorenstein smoothable $\mathbb Q$-Fano variety which admits a K\"ahler-Einstein metric. Recently a proper algebraic compactification $\overline {\mathcal M}$ of $\mathcal M$ was constructed by Odaka, Chi Li, et al. \cite{28},\cite{30} for proving the projectivity of $\overline{\mathcal M}$. 

There is a belief due to Tian saying that a moduli space with canonical metric is likely to be quasi-projective

Let
$W\subset \mathbb C^n$
be a domain, and $\Theta$ a positive current of degree $(q,q)$ on
$W$. For a point $p\in W$
one defines
$$\mathfrak v(\Theta,p,r)=\frac{1}{r^{2(n-q)}}\int_{|z-p|<r}\Theta(z)\wedge (dd^c|z|^2)^{n-q}$$
The
Lelong number
of $\Theta$ at
$p$
is defined as

$$\mathfrak v(\Theta,p)=\lim_{r \to 0}\mathfrak v(\Theta,p,r)$$

Let $\Theta$ be the curvature of singular hermitian metric $h=e^{-u}$, one has

$$\mathfrak v(\Theta,p)=\sup\{\lambda\geq 0: u\leq \lambda\log(|z-p|^2)+O(1)\}$$
see \cite{16}

 We set $K_{X/Y}=K_X\otimes \pi^*K_Y^{-1}$ and call it the relative canonical bundle of $\pi:X\to Y$

\begin{defn} Let $X$ be a K\"ahler variety with $\kappa(X)>0$ then the \textbf{relative K\"ahler-Einstein metric} is defined as follows 
$$Ric_{X/Y}^{h_{X/Y}^\omega}(\omega)=-\Phi\omega$$
where 
$$Ric_{X/Y}^{h_{X/Y}^\omega}(\omega)=\sqrt[]{-1}\partial\bar\partial\log(\frac{\omega^n\wedge\pi^*\omega_{can}^m}{\pi^*\omega_{can}^m})$$
and $\omega_{can}$ is a canonical metric on $Y=X_{can}$. Where $\Phi$ is the fiberwise constant function.

$$Ric_{X/Y}^{h_{X/Y}^{\omega_{SRF}}}(\omega)=\sqrt[]{-1}\partial\bar\partial\log(\frac{\omega_{SRF}^n\wedge\pi^*\omega_{can}^m}{\pi^*\omega_{can}^m})=\omega_{WP}$$
here $\omega_{WP}$ is a Weil-Petersson metric.

Note that if $\kappa(X)=-\infty$ then along Mori fibre space $f:X\to Y$ we can define Relative K\"ahler-Einstein metric as $$Ric_{X/Y}^{h_{X/Y}^\omega}(\omega)=\Phi\omega$$ when fibers and base are K-poly-stable, and we have fiberwise KE-stability
\end{defn}

\textbf{Remark 1}:From \cite{28}, by the same method we can show that, the log CM line bundle $L_{CM}^D$ descends to a line bundle $\Lambda_{CM}^D$ on the proper moduli space $\overline{\mathcal M}^D$. Moreover, there is a canonically defined continuous Hermitian metric $h_{DP}^D$, called as Deligne Hermitian metric on $\Lambda_{CM}^D$ whose curvature form is a positive current$\omega_{WP}^D$ and called logarithmic Weil-Petersson metric on compactified moduli space of K\"ahler-Einstein log Fano manifolds and it can be extended to canonical Weil-Petersson current $\dot\omega_{WP}^D$ on ${\mathcal M}^D$. Viehweg \cite{29}, showed that the moduli space of polarized manifolds with nef canonical line bundle is quasi-projective and by the same method of Chi Li \cite{28}, $\Lambda_{CM}^D$ on the proper moduli space $\overline{\mathcal M}^D$ is nef and big. Hence ${\mathcal M}^D$ is quasi-projective. We can prove it by using Lelong number method also which is more simpler.

In fact, if we prove that the Lelong number of singular hermitian metric 
$(L_{CM},h_{WP})$
corresounding to Weil-Petersson current on virtual CM-bundle has vanishing Lelong number, then 
$L_{CM}$ is nef and if we know the nefness of $L_{CM}$ then due to a result of E.Viehweg, we can prove that the moduli space of Kahler-Einstein Fano manifolds 
$M$
is quasi-projective and its compactification 
$M$
is projective. From Demailly's theorem, We know if 
$(L,h)$ be a positive , singular hermitian line bundle , whose Lelong numbers vanish everywhere. Then 
$L$
is nef. Gang Tian recently proved that, the CM-bundle 
$L_{CM}$
is positive,
So by the same theorem 2 and Poroposition 6, of the paper,\cite{42}
we can show that the Weil-Petersson current on mduli space of Fano Kahler-Einstein varieties has zero Lelong number(in smooth case it is easy, for singular case(for log terminal singularities ) we just need to show that the diameter of fibers are bounded ).

Moreover, by using the same method of Theorem 2, in \cite{42}, hermitian metric corresponding to $L_{CM}^D$ has vanishing Lelong number, and since from the recent result of Tian, $L_{CM}^D$ is positive, hence $L_{CM}^D$ is nef. So by the theorem of Viehweg \cite{29}, $\mathcal M^D$ is quasi-projective. In fact this give a new analytical method to the recent result of \cite{28}.

We have $$\omega_{WP}^D:=-\frac{\sqrt[]{-1}}{2\pi}\partial\bar\partial \log h_{DP}$$

where localle $h_{DP}=e^{-\Psi}$ and hence 
$$\omega_{WP}^D:=\frac{\sqrt[]{-1}}{2\pi}\partial\bar\partial \Psi$$
which $\Psi$ called Deligne potential.

Moreover, we have 

$$\omega_{WP}=Ric\left(\frac{\omega_{SKE}^\kappa\wedge \|v^*\|_{\left(i^*h_{FS}\right)^{\frac{1}{m}}}^{2/m}(v\wedge\bar v)^{1/m}}{\|v^*\|_{\left(i^*h_{FS}\right)^{\frac{1}{m}}}^{2/m}(v\wedge\bar v)^{1/m}\mid S\mid^2}\right)$$

Here semi-K\"ahler-Einstein metric $\omega_{SKE}$ is a $(1,1)$-current with log-pole singularities such that its restriction on each fiber is Kahler-Einstein metric. Note that from Kodaira-Spencer theory, we always have such semi-K\"ahler-Einstein metric.

\begin{defn} The null direction of fiberwise K\"ahler-Einstein metric $\omega_{SKE}$ gives a foliation along Mori-fibre space $\pi:X\to Y$ and we call it fiberwise K\"ahler-Einstein foliation and can be defined as follows $$\mathcal F=\{\theta\in T_{X/Y}|\omega_{SKE}(\theta,\bar\theta)=0\}$$
and along Mori-fibre space  $\pi:(X,D)\to Y$, we can define the following foliation $$\mathcal F'=\{\theta\in T_{X'/Y}|\omega_{SKE}^D(\theta,\bar\theta)=0\}$$
where $X'=X\setminus D$. In fact from Theorem 0,9. the Weil-Petersson metric $\omega_{WP}$ vanishes everywhere if and only if $\mathcal F=T_{X/Y}$. Note that such foliation may be fail to be as foliation in horizontal direction. But is is foliation in fiber direction.
\end{defn}

\;

We have the following result due to Bedford-Kalka \cite{49}

\textbf{Lemma:} Let $\mathcal L$ be a leaf of $f_*\mathcal F'$, then $\mathcal L$ is a closed complex submanifold and the leaf $\mathcal L$ can be seen as fiber on the moduli map $$\eta:\mathcal Y\to \mathcal M_{KE}^D$$ where $\mathcal M_{KE}^D$ is the moduli space of Kahler-Einstein fibers with at worst log terminal singularites and 

$$\mathcal Y=\{y\in Y_{reg}|(X_y,D_y)\; \; \text{is Kawamata log terminal pair}\}$$

\section{Main Theorems}

Now by applying Tian-Li \cite{22} method on Orbifold regularity of weak K\"ahler-Einstein metrics, we have the following theorem. We show that if general fibers and base are K-poly stable then there exists a generalized K\"ahler-Einstein metric on the total space which twisted with logarithmic Wil-Petersson metric and current of integration of special divisor, which is not exactly $D$. In fact along Minimal model $\pi:X\to X_{can}$, we can find such divisor $N$ by Fujino-Mori's higher canonical bundle formula. But along Mori-fibre space I don't know such formula to apply it in next theorem.

Let $f:X\to B$ be the surjective morphism of a normal projective variety $X$
of dimension $n=m+l$ to a nonsingular projective $l$-fold $B$ such that:

i) $(X,D)$ is sub klt pair.

ii) the generic fiber
$F$ of $\pi$ is a geometrically irreducible variety with vanishing log Kodaira dimension. We fix the smallest
$b\in \mathbb Z>0$ such that the

$$\pi_*\mathcal O_X(b(K_X+D))\neq 0$$

The Fujino-Mori \cite{25} log-canonical bundle formula
for $\pi:(X,D)\to B$ is 

$$b(K_X+D)=\pi^*(bK_B+L_{(X,D)/B}^{ss})+\sum_Ps_P^D\pi^*(P)+B^D$$
where $B^D$ is $\mathbb Q$-divisor on $X$ such that $\pi_*\mathcal O_X([iB_+^D])=\mathcal O_B$ ($\forall i>0$). Here $s_P^D:=b(1-t_P^D)$ where $t_P^D$ is the log-canonical threshold of $\pi^*P$ with respect to $(X,D-B^D/b)$ over the generic point $\eta_P$ of $P$. i.e., 

$$t_P^D:=\max \{t\in \mathbb R\mid \left(X,D-B^D/b+t\pi^*(P)\right)\;  \text{is sub log canonical over}\; \eta_P\}$$

Now if fibres $(X_s,D_s)$ are log Calabi-Yau varieties and the base $c_1(B)<0$, then we have

$$Ric(\omega_{X})=-\omega_{B}+\omega_{WP}^D+\sum_P(b(1-t_P^D))[\pi^*(P)]+[B^D]$$

Note that such pair of canonical metric is not unique.

Moreover, Let $\pi:(X,D)\to B$
is a holomorphic submersion onto
a compact K\"ahler manifold
$B$ with $B$ be a Calabi-Yau manifold, log fibers $(X_s,D_s)$ are log Calabi-Yau, and $D$ is a simple normal crossing divisor in $X$ with conic singularities.
Then $(X,D)$
admits a unique smooth
metric $\omega_B$
solving 

$$Ric(\omega_{can})=\omega_{WP}^D+\sum_P(b(1-t_P^D))[\pi^*(P)]+[B^D]$$
as current where $\omega_{can}$ has zero Lelong number and is good metric in the sense of Mumford.\cite{45}

\begin{thm} Let $\pi:(X,D)\to B$ be a holomorphic submersion between compact K\"ahler manifolds and $B$ and central fiber is $\mathbb Q$-Fano K-poly stable variety with klt singularities. Let $L+D$ be a relatively ample line bundle on $(X,D)$ such that the restriction of $c_1(L)$ to each fiber $(X_s,D_s)$ admits a K\"ahler-Einstein metric. Suppose that $\alpha-c_1(B)\geq 0$. Then, for all sufficiently large $r$, the adiabatic class

$$\kappa_r=c_1(L+D)+r\pi^*\kappa_B$$
contains a K\"ahler Einstein metric which is twisted by log Weil-Petersson current and current of integration of special divisor $N$ which is not $D$ in general. In fact such current of integration must come from higher canonical bundle formula along Mori-fibre space(I don't know such formula at the moment).

$$Ric(\omega_{(X,D)})=\omega_B+\omega_{WP}^D+(1-\beta)[N]$$

If $\alpha=c_1(B)$, $\kappa_B$ is any K\"ahler class on the base, whilst if $\alpha-c_1(B)>0$, then $\kappa_B=\alpha-c_1(B)$. 

\end{thm}

\begin{proof}
Let
$B$ be any $\mathbb Q$-Fano variety with klt singularities.Assume
$p\in B^{orb}$ is a quotient singularity where $B^{orb}$ is the orbifold locus of $X$. Then there exists a branched covering map $\tilde {\mathcal U_p}\to \tilde{\mathcal U_p}/G\cong \mathcal U_p$ where $\mathcal U_p$ is the small neighborhood of $p$. Now since $B$ is $\mathbb Q$-Fano variety, we have an embedding $i:B\to \mathbb P^n$ of linear system $|-mK_B|$ for $m\in\mathbb N$ sufficiently large and divisible.Now if $h_{FS}$ be the Fubini Study hermitian metric, then 

$$\omega=\sqrt[]{-1}\partial\bar\partial \log \left(i^*h_{FS}\right)^{-\frac{1}{m}}$$

gives a smooth positive $(1,1)$-form on regular part of $B$ and we show it with $B^{reg}$. but the fact is that on the singular locus
$B^{sing}$ the form $\omega$ in general is not canonically related to the local structure of $B$. So we must define new canonical measure for it. The following measure introduced by  Bo Berndtsson, 

$$\Omega=\|v^*\|_{\left(i^*h_{FS}\right)^{\frac{1}{m}}}^{2/m}(v\wedge\bar v)^{1/m}$$

Here
$v$
is any local generator of $\mathcal O(m(K_B+D))$ and $v^*$ is the dual generator of $\mathcal O(-m(K_B+D))$. 

Now consider the following Monge-Ampere equation on $B$.

$$(e^t\omega_{WP}+(1+e^t)\omega_0+\text{Ric} h+\sqrt[]{-1}\partial\bar\partial \phi)^n=\frac{e^{-\phi+\Psi}\|v^*\|_{\left(i^*h_{FS}\right)^{\frac{1}{m}}}^{2/m}(v\wedge\bar v)^{1/m}}{\|S\|^2}=e^{-\phi}\frac{\omega_{SKE}^\kappa\wedge \|v^*\|_{\left(i^*h_{FS}\right)^{\frac{1}{m}}}^{2/m}(v\wedge\bar v)^{1/m}}{\|v^*\|_{\left(i^*h_{FS}\right)^{\frac{1}{m}}}^{2/m}(v\wedge\bar v)^{1/m}\mid S\mid^2}$$

Where $\Psi$ is Deligne potential on log Tian's line bundle $L_{CM}^D$ and $\omega_{SKE}$ is called fiberwise K\"ahler-Einstein metric. Note that $diam(X_s,\omega_s)\leq C$ if and only of central fiber be K-poly stable with klt singularities , so we can get $C^0$-estimate. Note that we are facing with three type Complex-Monge-Ampere equation, 1) MA equation on fiber direction, 2) MA equation on horizontal direction , and 3) MA foliation when fiber-wise Kahler-Einstein metric iz zero, which the main difficulty for solving estimates is about Monge-Ampere equation corresponding to fiberwise K\"ahler-Einstein foliation. I don't know it yet. It is supper difficult!

Take a resolution
$\pi:(\tilde B,\tilde D)\to (B,D)$ with a simple normal crossing exceptional divisor $E=\pi^{-1}(B^{sing})=\cup_{i=1}^rE_i\bigcup \cup_{j=1}^sF_i$ where by klt property we have $a_i>0$ and $0<b_j<1$. Then we have 

$$K_{\tilde B}+\tilde D=\pi^*(K_B+D)+\sum_{i=1}^r a_iE_i-\sum_{j=1}^sb_jF_j$$

So by choosing a smooth K\"ahler metric $\tilde \omega$
on $\tilde B$, there exists $f\in C^\infty(\tilde B)$ such that

$$\pi^*(\frac{\Omega_{X/B}}{\|S\|^2})=e^f\frac{\prod_{i=1}^r\|\theta_i\|^{2a_i}}{\|\tilde S\|^2\prod_{j=1}^s\|\sigma_j\|^{2b_i}}\tilde \omega^n$$

where $\theta_i$, $\sigma_j$ and $\tilde S$ are defining sections of $E_i$, $F_j$ and $\tilde D$ respectively. Then we can pull back our Monge-Ampere equation to $(\tilde{B} ,\tilde{D})$ and by taking $\omega=e^t\omega_{WP}+(1+e^t)\omega_0$

$$(\pi^*\omega+\text{Ric} \tilde h+\sqrt[]{-1}\partial\bar\partial \psi)^n=e^{f-\psi+\Psi}\frac{\prod_{i=1}^r\|\theta_i\|^{2a_i}}{\|\tilde S\|^2\prod_{j=1}^s\|\sigma_j\|^{2b_i}}{\tilde \omega^n}$$
Now by applying Theorem 0.8, if we take $\psi_+=f+\sum_ia_i\log |\theta_i|^2$ and $\psi_-=\psi-\Psi+\sum_jb_j\log|\sigma_j|^2+\sum_l\log c_l|\tilde s_l|^2$, where $\tilde S=\sum_l s_l\tilde S_l$
then we have 

$$(\pi^*\omega+\text{Ric} \tilde h+\sqrt[]{-1}\partial\bar\partial \psi)^n=e^{\psi_+-\psi_-}{\tilde \omega^n}$$

It’s easily can be checked they satisfy the quasi-plurisubharmonic condition:

$$\sqrt[]{-1}\partial\bar\partial \psi_+\geq -C{\tilde \omega^n}\;,\;\; \sqrt[]{-1}\partial\bar\partial \psi_-\geq -C{\tilde \omega^n}$$ for some uniform constant
$C >0$. By regularizing our Monge-Ampere equation away from singularities we have

$$(\omega_\epsilon+\epsilon\text{Ric} \tilde h+\sqrt[]{-1}\partial\bar\partial \psi_\epsilon)^n=e^{\psi_{+,\epsilon}-\psi_{-,\epsilon}}{\tilde \omega_\epsilon^n}$$ where $\omega_\epsilon=\pi^*\omega-\epsilon\sqrt[]{-1}\partial\bar\partial |S_E|^2$ is a K\"ahler metric on $\tilde B\setminus \tilde D$. The Laplacian estimate for the solution $\omega_\epsilon$ away from singularities is due to M. Paun \cite{21} and Demailly-Pali\cite{26}(a more simpler proof). So we have $C^{1,\alpha}$ estimate for our solution. So,  we get $$Ric(\omega_X)=\omega_B+\omega_{WP}^D+(1-\beta)[N]$$.
\end{proof}

\textbf{Remark 2}: A $\mathbb Q$-Fano variety $X$ is K-stable if and only if it be K-stable with respect to the central fibre of test configuration. Hence $X$ is a  $\mathbb Q$-Fano K-stable if and only if 

$$\int_{X_y}\|v^*\|_{\left(i^*h_{FS}\right)^{\frac{1}{m}}}^{2/m}(v\wedge\bar v)^{1/m}\leq C$$

Here
$v$
is any local generator of $\mathcal O(m(K_{X_y}))$ and $v^*$ is the dual generator of $\mathcal O(-m(K_{X_y}))$ and $C$ is a constant and $X_t$ is a general fibre with respect to test configuration. 

We give the following conjecture about K-stability via Weil-Petersson geometry

\textbf{Conjecture}: A Fano variety $X$ is K-stable if and only if 
$0$ is at finite Weil-Petersson distance from $C^{o}$ i.e. $$d_{WP}(C^{0},0)<+\infty$$ where here $C$ is just base of test configuration

Now we introduce relative Tian's alpha invariant along Fano fibration and Mori fibre space which correspounds to twisted K-stability.

Tian in \cite{5}, obtained a suficient condition to get $C^0$-estimates for Monge-Ampere equation related to K\"ahler-Einstein metric for Fano varieties. It involves an invariant of the manifold and called the Tian's $\alpha$-invariant, that encodes the possible singularities of singular hermitian metrics with non negative curvature on $K_M$. The suficient condition for the existence of K\"ahler-Einstein metric remain the same. One can also define the $\alpha$-invariant for any ample line bundle on a complex manifold $M$.

\begin{defn}Let $M$ be an $n$-dimensional compact K\"ahler manifold with an ample line bundle $L$. We fix $\omega$ a K\"ahler metric in $c_1(L)$, and define Tian’s  $\alpha$-invariant

$$\alpha(L)=\sup\{\alpha>0|\; \exists C>0\;  \text{with}\; \int_Me^{-\alpha(\varphi-\sup_M\varphi)}\omega^n\leq C\}$$
where we have assumed $\varphi\in C^{\infty}$ and $\omega+\sqrt[]{-1}\partial\bar\partial\varphi>0$

\end{defn}

Now we have the following remark as relative $\alpha$-invariant on the degeneration of Fano K\"ahler-Einstein metrics.

\textbf{Remark 3}:  Let $\pi:X\to B$ be a holomorphic submersion between compact K\"ahler manifolds whose fibers and the base admit no non-zero holomorphic vector fields and $B$ is Fano variety. Let each Fano fiber $X_s$ admits a K\"ahler-Einstein metric. Define the relative $\alpha$-invariant 
$$\alpha(X/B)=\sup\{\alpha>0|\; \exists C>0\;  \text{with}\; \int_{X/B} e^{-\alpha(\varphi-\sup_M\varphi)}\frac{\omega_{SKE}^\kappa\wedge \|v^*\|_{\left(i^*h_{FS}\right)^{\frac{1}{m}}}^{2/m}(v\wedge\bar v)^{1/m}}{\|v^*\|_{\left(i^*h_{FS}\right)^{\frac{1}{m}}}^{2/m}(v\wedge\bar v)^{1/m}\mid S\mid^2}\leq C\}$$
and let $$\alpha(X/B)>\frac{n}{n+1}$$

then we have relative K\"ahler Einstein metric 
$$Ric_{X/B}(\omega)=\Phi\omega$$

where $\Phi$ is the fiberwise constant function.

\subsection{Relative K\"ahler Ricci soliton}
In this subsection, we extend the notion of K\"ahler Ricci soliton to relative K\"ahler Ricci soliton along holomorphic fibre space $\pi:X\to B$. We think that the notion of relative K\"ahler Ricci soliton is more suitable for Song-Tian program along Mori fibre space or in general for Fano fibration.

A K\"ahler-Ricci soliton is one of the generalization of a K\"ahler-Einstein metric and closely related
to the limiting behavior of the normalized K\"ahler-Ricci flow.

A K\"ahler metric $h$ is called a K\"ahler-Ricci soliton if its K\"ahler form $\omega_h$
satisfies equation

$$Ric(\omega_h)-\omega_h=L_X\omega_h=d(i_X\omega_h)$$

where $Ric(\omega_h)$ is the Ricci form of
$h$
and
$L_X\omega_h$
denotes the Lie derivative of $\omega_h$
along a holomorphic vector field $X$ on $M$.

Since $i_X\omega_h$ is $(0,1)$ $\bar\partial $-closed form, by using Hodge theory, $$i_X\omega_h=\alpha+\bar\partial\varphi$$
where $\alpha$ is a $(0,1)$-harmonic form and $\varphi\in \mathcal C^\infty (M,\mathbb C)$ hence, we have
$$Ric(\omega_h)-\omega_h=\partial\bar\partial \varphi$$

we can find a smooth real-valued function $\theta_X(\omega)$ such that
$$i_X\omega=\sqrt[]{-1}\bar\partial \theta_X(\omega)$$

This means that if the K\"ahler metric $h$ be a K\"ahler-Ricci soliton, then $\omega_h\in c_1(M)$

\textbf{Theorem:} Let $X$ and $B$ are compact K\"ahler manifolds with holomorphic submersion $\pi:X\to B$, and let fibers $\pi^{-1}(b)=X_b$ are K-poly stable Fano varieties and the Fano variety $B$ admit K\"ahler Ricci soliton. Then there exists a generalized K\"ahler Ricci soliton 
$$Ric(\omega_X)=\omega_B+\pi^*\omega_{WP}+L_X\omega_B$$
for some holomorphic vector field on $X$, where $\omega_{WP}$ is a Weil-Petersson metric on moduli space of K-poly stable Fano fibres.

\textbf{Idea of Proof}. It is enough to solve the following Monge-Ampere equation 

$$(\omega+\sqrt[]{-1}\partial\bar\partial\varphi)^n=e^{\Psi-\varphi-\theta_X-X(\varphi)}\omega^n$$

where $\Psi$ is the Deligne potential on Tian's CM-line bundle $L_{CM}$.  We need to show properness of the relative K-energy. We postpone it for future work.

Moreover, we have the same result on pair $(X,D)\to B$. i.e.
$$Ric(\omega_X)=\omega_B+\pi^*\omega_{WP}^D+L_X\omega_B+[N]$$

We define the relative Soliton-K\"ahler Ricci flow as 
$$\frac{\partial \omega(t)}{\partial t}=-Ric_{X/Y}(\omega(t))+\Phi\omega(t)+L_X\omega(t)$$
along holomorphic fiber space $\pi:X\to Y$, where here the relative Ricci form $Ric_{X/Y,\omega}$ of $\omega$ is defined by

$$Ric_{X/Y,\omega}=-\sqrt[]{-1}\partial\bar\partial\log (\omega^m\wedge \pi^*|dy_1\wedge dy_2\wedge...\wedge dy_k|^2)$$
where $(y_1,...,y_k)$ is a local coordinate of $Y$.

Note that if fibers $X_y$ be K-poly stable, base $Y$ and central fiber admit K\"ahler Ricci soliton then relative Soliton-K\"ahler Ricci flow flow converges to 

$$Ric(\omega_X)=\omega_B+\pi^*\omega_{WP}+L_X\omega$$

A relative K\"ahler metric
$g$ on holomorphic fibre space 
$\pi:M\to B$ where $M$ and $B$ are compact K\"ahler manifolds
is called a
relative K\"ahler-Ricci soliton if there is a relative holomorphic vector field
$X$
on the relative tangent bundle
$T_{M/B}$
such that the relative K\"ahler form
$\omega_g$
of
$g$
satisfies

$$Ric_{M/B}\omega_g-\Phi\omega_g=L_X\omega_g$$

So we can extend the result of Tian-Zhu for future work. i.e. If we have already relative K\"ahler Ricci soliton $\omega_{RKS}$, then the solutions of the relative K\"ahler Ricci flow

$$\frac{\partial \omega(t)}{\partial t}=-Ric_{X/Y}(\omega(t))+\Phi\omega(t)$$

converges to $\omega_{RKS}$ in the sense of Cheeger-Gromov with additional assumption on the initial metric $g_0$ .

We give a conjecture about invariance of plurigenera using K-stability and K\"ahler Ricci flow.

\begin{thm}(Siu \cite{32}) Assume
$\pi: X\to B$ is smooth, and every
$X_t$ is of general type.
Then the plurigenera $P_m(X_t)=\dim H^0(X_t,mK_{X_t})$ is independent of $t\in B$ for any $m$.\end{thm}
See \cite{31, 33,34} also for singular case.

\textbf{Conjecture}:Let $\pi: X\to B$ is smooth, and every
$X_t$ is K-poly stable.
Then the plurigenera $P_m(X_t)=\dim H^0(X_t,-mK_{X_t})$ is independent of $t\in B$ for any $m$.

Idea of proof. We can apply the relative K\"ahler Ricci flow method for it. In fact if we prove that $$\frac{\partial \omega(t)}{\partial t}=-Ric_{X/Y}(\omega(t))+\Phi\omega(t)$$
has long time solution along Fano fibration such that the fibers are K-poly stable then we can get the invariance of plurigenera in the case of K-poly stability

\textbf{Conjecture}: The Fano variety $X$ is K-poly stable if and only if for the proper holomorphic fibre space $\pi:X\to \mathbb D$  which the Fano fibers have unique K\"ahler-Einstein metric with positive Ricci curvature, then the fiberwise K\"ahler Einstein metric (Semi-K\"ahler Einstein metric) $\omega_{SKE}$ be smooth and semi-positive. Note that if fibers are K-poly stable then by Schumacher and Berman \cite{35,36} result we have 
$$-\Delta_{\omega_t} c(\omega_{SKE})-c(\omega_{SKE})=|A|_{\omega_t}^2$$
where $A$ represents the Kodaira-Spencer class of the deformation and since $\omega_{SKE}^{n+1}=c(\omega_{SKE})\omega_{SKE}^nds\wedge d\bar s$ so $c(\omega_{SKE})$ and $\omega_{SKE}$ have the same sign. By the minimum principle $\inf \omega_{SKE}<0$. But our conjecture says that the fibrewise Fano K\"ahler-Einstein metric $\omega_{SKE}$ is smooth and semi-positive if and only if $X$ be K-poly stable.

Chi Li in his thesis showed that an anti-canonically polarized Fano variety is K-stable if and only if it is K-stable with respect to test configurations with normal central fibre. So if the central fibre $X_0$ admit Kahler-Einstein metric with positrive Ricci curvature then along Mori-fibre space all the general fibres $X_t$ admit K\"ahler-Einstein metric with positive Ricci curvature(we will prove it in this paper) and hence we can introduce fibrewise K\"ahler-Einstein metric $\omega_{SKE}$. See \cite{35,36}

Let $L \to X$ be a holomorphic line bundle over a complex manifold $X$ and fix an open cover $ X =
\cup U_\alpha $for which there exist local holomorphic frames $e_\alpha : U_\alpha \to L$. The transition functions $g_{\alpha\beta}=e_\beta/e_\alpha\in \mathcal O_X^*(U_\alpha\cap U_\beta)$ determine the Cech 1-cocycle $\{(U_\alpha,g_{\alpha\beta})\}$. If $h$ is a singular Hermitian metric on $L$ then $h(e_\alpha,e_\alpha)=e^{-2\varphi_\alpha}$ , where the functions $\varphi_\alpha\in L_{loc}^1(U_\alpha)$ are called the local weights of the metric $h$.
We have $\varphi_\alpha = \varphi_\beta + \log |g_{\alpha\beta}|$ on $U_\alpha\cap U_\beta$ and the curvature of $h$, $$c_1(L,h)|_{U_\alpha}=dd^c\varphi_\alpha$$ is a well defined closed $(1,1)$ current on $X$.

One of the important example of singular hermitian metric is singular hermitian
metric with algebraic singularities. Let $m$ be a positive integer and $\{S_i\}$ a finite number of global holomorphic sections of $mL$. Let $\varphi$ be a $C^{\infty}$-function on $X$. Then

$$h:=e^{-\varphi}.\frac{1}{(\sum_i|S_i|^2)^{1/m}}$$
defines a singular hermitian metric on $L$. We call such a metric $h$ a singular hermitian
metric on $L$ with algebraic singularities.

\textbf{Remark 4}:Fiberwise K\"ahler-Einstein metric $\omega_{SKE}$ along Mori fiber space(when fibers are K-poly stable) has non-algebraic singularities. In fact such metric $\omega_{SKE}$ has pole singularites

\textbf{Remark 5}:Let $f\colon X\to Y$ be a surjective proper holomorphic fibre space such that $X$ and $Y$ are projective varieties and central fibre $X_0$ is Calabi-Yau variety with canonical singularities, then all the fibres $X_t$ are also Calabi-Yau varieties. Let for simplicity that $Y$ is a smooth curve. Since $X_0$ has canonical, then, we may assume that $X$ is canonical. Nearby fibers are then also canonical. $O_X(K_X+X_0)\to O_{X_0}(K_{X_0})$ is surjective and so if $K_{X_0}$ is Cartier, then so is $K_X$ on a neighborhood of $X_0$.We have that $P_m(X_y)$ is deformation invariant for $m\geq 1$ so in fact $h^0(K_{X_y})>0$ for $y\in Y$. Since $K_{X_0}\equiv 0$, then $K_{X_y}\equiv 0$ and so $K_{X_y}\sim 0$

Note that if all of the fibres of the surjective proper holomorphic fibre space $f\colon X\to Y$ are Calabi-Yau varieties, then the central fiber $X_0$ may not be Calabi-Yau variety

\textbf{Remark 6}:Let $f\colon X\to Y$ be a surjective proper holomorphic fibre space such that $X$ and $Y$ are projective varieties and central fibre $X_0$ is of general type, then all the fibres $X_t$ are also of general type varieties. Note that if base and fibers are of general type then the total space is of general type

\textbf{Remark 7}:Let $f\colon X\to Y$ be a surjective proper holomorphic fibre space such that $X$ and $Y$ are projective varieties and central fibre $X_0$ is K-poly stable with discrete Automorphism group, then all the fibres $X_t$ are also K-poly stable. Note that if along Mori-fiber space, the general fibers and base space are of K-poly stable, then total space is also K-poly stable. Note that Mori fiber space is relative notion and we can assume that base is K-poly stable always.

Define $$\text{Scal}(\omega)=\frac{n Ric(\omega)\wedge \omega^{n-1}}{\omega^n}$$ and take a holomorpic submerssion $\pi:X\to D$ over some disc $\mathbb D$. It is enough to show smoothness of semi Ricci flat metric over a disc  $\mathbb D$. Define a map $F:\mathbb D\times C^\infty (\pi^{-1}(b))\to C^\infty (\pi^{-1}(b))$ by 
$$F(b, \rho)=Scal(\omega_0|_{b}+\sqrt[]{-1}\partial_b\bar{\partial_b}\rho)$$

We can extend $F$ to a smooth map $\mathbb D\times L_{k+4}^{2}(\pi^{-1}(b))\to L_k^2(\pi^{-1}(b))$. From the definition $F(b,\rho_b)$ is constant. 

Recall that, if
$L$
denotes the linearization of $\rho\to Scal(\omega_\rho)$, then 
$$L(\rho)=\mathcal D^*\mathcal D(\rho)+\nabla Scal.\nabla\rho$$
where $\mathcal D$ is defined by
$$\mathcal D=\bar \partial \text{o}\nabla:C^\infty\to \Omega^{0,1}(T)$$

and
$\mathcal D^*$
is
$L^2$
adjoint of
$\mathcal D$. The linearisation of
$F$
with respect to $\rho$
at $b$ is given by $$\mathcal D_b^*\mathcal D_b:C^\infty\left({\pi^{-1}(b)}\right)\to C^\infty\left({\pi^{-1}(b)}\right)$$

Leading order term of $L$ is $\Delta^2$. So
$L$ is elliptic and we can use the standard theory of elliptic partial differential equations. Since central fiber has discrite Automorphism, hence Kernel of such Laplace-Beltrami is zero, and hence $\mathcal D_b^*\mathcal D_b$ is an isomorphism. By the implicit
function theorem, the map $b\to \rho_b$ is a smooth map

$$\mathbb D\to L_k^2(\pi^{-1}(b))\;\;; \forall k.$$ By
Sobolev embedding, it is a smooth map
$D\to \mathbb C^r(\pi^{-1}(b))$ for any $r$. Hence cscK or K\"ahler-Einstein metric with positive Ricci curvature on the central fiber can be extended smoothly to general fibers. See \cite{43,44}

\;
\;

In final it is worth to mention that K-stability along special test configuration gives such foliation as we introduced and Tian , Chen-Sun-Donaldson didn't consider such Monge-Ampere foliation for their solution of Yau-Tian-Donaldson, conjecture.

\end{document}